\documentclass[12pt, reqno]{amsart}
\usepackage{epsfig, graphicx, cite, amssymb, amsthm, color}
\usepackage[english]{babel}
\usepackage{amsfonts,amscd, mathrsfs,enumerate}
\usepackage{geometry}
\usepackage{hyperref}
\usepackage{url,color,srcltx,mathrsfs}

\begin{document}
\def\K{\mathbb{K}}
\def\R{\mathbb{R}}
\def\C{\mathbb{C}}
\def\Z{\mathbb{Z}}
\def\Q{\mathbb{Q}}
\def\D{\mathbb{D}}
\def\N{\mathbb{N}}
\def\T{\mathbb{T}}
\def\P{\mathbb{P}}
\def\A{\mathscr{A}}
\def\CC{\mathscr{C}}
\renewcommand{\theequation}{\thesection.\arabic{equation}}
\newtheorem{theo}{Theorem}[section]
\newtheorem{lemma}{Lemma}[section]
\newtheorem{coro}{Corollary}[section]
\newtheorem{prop}{Proposition}[section]
\newtheorem{definition}{Definition}[section]
\newtheorem{remark}{Remark}[section]

\newtheorem{example}{Example}[section]
\newtheorem{notation}{Notation}
\newtheorem{con}{Consequence}
\bibliographystyle{plain}
\theoremstyle{plain}

\title[Solving the $\partial \bar{\partial}$ with prescribed support~]{\textbf{Solving the $\partial \bar{\partial}$ with prescribed support}}
\author[W.\ O.\  Ingoba  \& S.\  Sambou ]{ Winnie Ossete Ingoba \& Souhaibou Sambou }
\address{ Departement of Mathematics\\ Faculty of Science and Technology \\ University Marien Ngouabi, (Congo)}
\email{wnnossete@gmail.com}
\address{ Departement of Mathematics\\  UFR of Applies Sciences and Technology \\  University Gaston Berger  of Saint-Louis,  BP: 234 (Senegal)}
\email{souhaibou.sambou@ugb.edu.sn}
%\address{Departement of Mathematics\\UFR of Sciences and Technologies \\ University Assane Seck  of Ziguinchor, BP: 523 (Senegal)}
%\email{s.diatta1375@zig.univ.sn }
%\author[S.\  Sambou ]{ Souhaibou Sambou }
%\address{Departement of Mathematics\\UFR of Sciences and Technologies \\  University Assane Seck  of Ziguinchor, BP: 523 (Senegal)}
%\email{s.sambou1440@zig.univ.sn }
%\author[E.\  Bodian ]{ Eramane Bodian }
%\address{Departement of Mathematics\\UFR of Sciences and Technologies \\  University Assane Seck  of Ziguinchor, BP: 523 (Senegal)}
%\email{m.bodian@univ-zig.sn }
\author[S.\  Sambou ]{ Salomon Sambou }
\address{Departement of Mathematics\\UFR of Sciences and Technology \\ University Assane Seck  of Ziguinchor, BP: 523 (Senegal)}
\email{ssambou@univ-zig.sn }
%\author[S.\  Khidr ]{ Shaban Khidr }
%\address{Departement of Mathematics\\Faculty of Science\\University of Jeddah\\21589 Jeddah, Saudi Arabia}
%\email{skhidr@uj.edu.sa}

\maketitle
%%%%%%%%%%%%%%%%%%%%%%%%%%%%%%%%%%%%%%%%%%%%%%%%%%%%%%%%%%%%%%%%%%%%%%%%%%%%%%%%%%%%%%%%%%%%%%%%%%%%%%%

%%%%%%%%%%%%%%%%%%%%%%%%%%%%%%%%%%%%%%%%%%%%%%%%%%%%%%%%%%%%%%%%%%%%%%%%%%%%%%%%%%%%%%%%%%%%%%%%%%%%%%%%%%
%%%%%%%%%%%%%%%%%%%%%%%%%%%%%%%%%%%%%%%%%%%%%%%%%%%%%%%%%%%%%%%%%%%%%%%%%%%%%%%%%%%%%%%%%%%%%%%%%%%%%%%%%% 
\begin{abstract}

In this paper, we consider the problem of solving the $\partial \bar{\partial}$ with prescribed support for forms or currents in a domain $\Omega$  of an complex manifold $X$.
%%%%%%%%%%%%%%%%%%%%%%%%%%%%%%%%%%%%%%%%%%%%%%%%%%%%%%%%%%%%%%%%%%%%%%%%%%%%%%%%%%%%%%%%%%%%%%%%%%%%%%%%%%
\vskip 2mm
\noindent
\vskip 1.3mm
\noindent
%\textit{
 
\end{abstract}

\section{Introduction}
Let $X$ be an complex manifold of complex dimension $n$ and $\Omega \subset \subset X$ be a domain. In this paper, we will consider the problem of solving the $\partial \bar{\partial}$ with prescribed support for forms or currents. The motivation of this work is to establish the analogue of the  Proposition $1.1$ of \cite{TS} for the $\partial \bar{\partial}$ operator. We take some consequences on the solving with prescribed support of the $(1,1)$-Bott-Chern cohomology group related with the vaniching of the De Rham and Dolbeault cohomology groups.  We ask the following question:\\
Let $T$ be a $(1,1)$-current or a $(1,1)$-differential form of class $C^\infty$ (Resp $L^p_{loc}$, $C^k$ ) $d$-closed  on $X$ with support in $\overline{\Omega}$, is there a distribution $S$ or a function $S$ of class $C^\infty$ (Resp $L^p_{loc}$, $C^k$) with support in $\overline{\Omega}$ such that $\partial \bar{\partial}S = T$? \\
To answer this question, we start by giving some properties of the support of the solution of the $\partial \bar{\partial}$ equation if it exists. We then give an extension results for pluriharmonic functions under certain vaniching hypotheses of the De Rham cohomology groups $H^2_{\overline{\Omega}, cour}(X)$ (Resp $H^2_{\overline{\Omega}, L^p_{loc}}(X)$, $H^2_{\overline{\Omega}, C^k}(X)$) and the Dolbeault $H^{0,1}_{\overline{\Omega}, cour}(X)$ (Resp $H^{0,1}_{\overline{\Omega}, L^p_{loc}}(X)$, $H^{0,1}_{\overline{\Omega}, C^k}(X)$). Finally, we answer the question with the hypotheses $H^2(X) = H^{0,1}(X) = 0$ and the extension of pluriharmonic functions. As a last result, we have the solution of the $\partial \bar{\partial}$ equation with compact support for $(1,1)$-currents $d$-closed with as hypotheses $H^2_c(X) = H^{0,1}_c(X) = 0$.
\section{Definition and Notations}
\begin{definition}{(Cf \cite{WG})}
A complex manifold $X$ is called cohomologically $q$-complete if its $k$-th cohomology group with coefficients in $\mathcal{F}$, $H^k(X, \mathcal{F})$, vanishes for every coherent analytic sheaf $\mathcal{F}$ over $X$ and integer $k \geq q$.

\end{definition}
We denote by $H^2_{\overline{\Omega}, cour}(X)$ (Resp $H^2_{\overline{\Omega}, L^p_{loc}}(X)$, $H^2_{\overline{\Omega}, C^k}(X)$) the $2$-th De Rham cohomology group of currents (Resp of differential forms of class $C^k$ or $L^p_{loc}$) with support in $\overline{\Omega}$, $H^2_c(X)$ the $2$-th De Rham cohomology group of differential forms with compact support. And by $H^{0,1}_{\overline{\Omega}, cour}(X)$ (Resp $H^{0,1}_{\overline{\Omega}, L^p_{loc}}(X)$, $H^{0,1}_{\overline{\Omega}, C^k}(X)$) the $(0, 1)$-th Dolbeault cohomology group of currents (Resp of differential forms of class $C^k$ or $L^p_{loc}$) with support in $\overline{\Omega}$, $H^{0,1}_c(X)$ the $(0,1)$-th Dolbault cohomology group of differential forms with compact support.
\section{Properties of the support and uniqueness of the solution}
Let $X$ be a complex manifold and $\Omega$ be a domain such that $\overline{\Omega}$ is strictly included in $X$.
 Therefore $\Omega^c = X \setminus \overline{\Omega}$ is a non-empty open set of $X$. As a first result, we have the following proposition which is the analogue for the $\partial \bar{\partial}$ of Proposition $1.1$ of \cite{TS}.
\begin{prop} \label{e}
Let $T$ be a $(1,1)$-current $\partial \bar{\partial}$-exact on $X$. If $\Omega^c$ is a connected component of $X \setminus SuppT$ and if $S$ is a distribution on $X$ such that $\partial \bar{\partial} S = T$, then $Supp S \cap \Omega^c = \emptyset$ or $\Omega^c \subset Supp S$.
\end{prop}
\begin{proof}
We have $\partial \bar{\partial} S = T$, so $S$ is a pluriharmonic distribution on $X \setminus SuppT$ and therefore a pluriharmonic function on $X \setminus SuppT$ and in particular on $\Omega^c$. Suppose that the support of $S$ does not contain $\Omega^c$. Then we distinguish two cases: either $\Omega^c\setminus SuppS= \emptyset$, then $Supp S \cap \Omega^c = \emptyset$ or $\Omega^c\setminus SuppS \neq \emptyset$. Suppose $\Omega^c\setminus SuppS \neq \emptyset$ then $S = 0$ on $\Omega^c\setminus SuppS$ which is an open set of $\Omega^c$. According to the principle of analytic extension for an analytic real function,
 $S$  vanishes on $\Omega^c$. So $Supp S \cap \Omega^c = \emptyset$ which is absurd.
 \end{proof}
 As the consequence of Proposition \ref{e}, we have
 \begin{coro} \label{z}
 Let $T$ be an $(1,1)$-current $\partial \bar{\partial}$-exact on $X$. Suppose $X$ is connected. Then if $S$ is a distribution on $X$ such that $\partial \bar{\partial} S = T$, then $Supp S = Supp T$ or $ Supp S= X$.
 \end{coro}
 \begin{proof}
 We have $ Supp T \subset Supp S$. If $Supp S \neq X$, then by Proposition \ref{e},\\ $Supp S \cap X \setminus Supp T = \emptyset$ which implies $Supp S \subset Supp T$ so $Supp S = Supp T$ or\\ $X \setminus Supp T \subset Supp S$ which implies $Supp S = X$.
\end{proof}
Note that the difference between two solutions of the equation $\partial \bar{\partial} S = T$ is weakly pluriharmonic. Then we have the following result:
\begin{prop}
Assume that the complex manifold $X$ is connected. Let $T$ be a $(1,1)$-current $\partial \bar{\partial}$-exact on $X$ such that $X\setminus SuppT \neq \emptyset$. Let $S$ and $U$ be two distributions such that $$\partial \bar{\partial} S = \partial \bar{\partial} U = T$$ and there exists a connected component $\Omega^c$ of $X\setminus SuppT$ such that $$Supp S \cap \Omega^c = Supp U \cap \Omega^c = \emptyset.$$
Then $S = U$.
In particular, the equation $\partial \bar{\partial} S =T$ has a unique solution $S$ such that $Supp S = Supp T$.
\end{prop}
\begin{proof}
We have $\partial \bar{\partial} S = T$ and $\partial \bar{\partial} U = T$, so $S-U$ is a pluriharmonic distribution on $X$ and therefore a pluriharmonic function on $X$ and in particular on $\Omega^c$. According to the principle of analytic extension for an analytic real function,
 $S-U$ is vanishes on $X$. So the equation $\partial \bar{\partial} S =T$ has a unique solution $S$. According to the Corollary \ref{z}, $Supp S = Supp T$.
 \end{proof}
 \section{solving the $\partial \bar{\partial}$ with prescribed support}
 Let $X$ be a a complex manifold and $\Omega$ a domain such that $\overline{\Omega}$ is strictly included in $X$. 
 Therefore $\Omega^c = X \setminus \overline{\Omega}$ is a non-empty open of $X$.
 \begin{prop} \label{a}
 Assume that $$H^2_{\overline{\Omega}, cur}(X) = H^{0,1}_{\overline{\Omega}, cur}(X) = 0.$$ Then any pluriharmonic function on $\Omega^c$ which is the restriction on $\Omega^c$ of a distribution on $X$ extends into a pluriharmonic function on $X$.
 \end{prop}
 \begin{proof}
 Let $f$ be a pluriharmonic function on $\Omega^c$ and $S_f \in D'(X)$ a distribution such that $S_{f_{\vert \Omega^c}} = f$. let $$g = \partial \bar{\partial}S_f.$$
 We have $g$ is a $(1,1)$-current, $d$-closed and with support in $\overline{\Omega}$. Since $H^2_{\overline{\Omega}, cur}(X) =0$, there exists a $1$-current $u$ with support in $\overline{\Omega}$ such that $du =g$. We have $u= u_1+ u_2$ where $u_1$ and $u_2$ are respectively a $(0,1)$-current with support in $\overline{\Omega}$ and a $(1,0)$-current with support in $\overline{\Omega}$. We have $du= du_1 +du_2=g$.
 Since $d= \partial + \bar{\partial}$ and for reasons of bidegre we have $\partial u_2 =0$ and $\bar{\partial} u_1 =0$. Since $H^{0,1}_{\overline{\Omega}, cur}(X) = 0$, there exist a distributions $w_1$ and $w_2$ with support in $\overline{\Omega}$ such that $u_1 = \bar{\partial}w_1$ and $u_2= \partial w_2$. We have
 \begin{center}
 $g = \partial u_1+ \bar{\partial}u_2= \partial \bar{\partial}w_1 +\bar{\partial}\partial w_2 = \partial \bar{\partial}(w_1 - w_2)$.
 \end{center}
 Let $v= w_1 - w_2$, $v$ is an distribution with support in $\overline{\Omega}$ such that $\partial \bar{\partial} v = g$.
 Let $h = S_f - v$, we have $\partial \bar{\partial} h = 0$ and $h_{\vert \Omega^c} = f$.
 \end{proof}
 We have the $L^p_{loc}$ version of Proposition \ref{a}
\begin{prop}
Assume that $$H^2_{\overline{\Omega}, L^p_{loc}}(X) = H^{0,1}_{\overline{\Omega}, L^p_{loc}}(X) = 0. $$ Then any pluriharmonic function on $\Omega^c$ which is the restriction on $\Omega^c$ of a function $L^p_{loc}(X)$ extends to a pluriharmonic function on $X$.
\end{prop}
We have the $C^k$ version with $k \geq 0$ of Proposition \ref{a}
\begin{prop}
Assume that $$H^2_{\overline{\Omega}, C^k}(X) = H^{0,1}_{\overline{\Omega}, C^k}(X) = 0 \mbox{ with }  k \geq 0. $$ Then any pluriharmonic function on $\Omega^c$ which is of class $C^{k + 1}$ on $\overline{\Omega^c}$ extends to a pluriharmonic function on $X$.
\end{prop}
\begin{remark}
If $X$ is a cohomologically $(n-1)$-complete  complex manifold of dimension $n \geq 2$, $\Omega \subset X$ a domain and $K \subset \Omega$ a compact such that $\Omega \setminus K$ is connected. Then according to the Theorem $1.1$ of \cite{WG}, any pluriharmonic function on $\Omega \setminus K$ is the restriction of a pluriharmonic function on $\Omega$.
\end{remark}
As a consequence of this result, we have
\begin{theo} \label{b}
Let $X$ be a cohomologically $(n-1)$-complete complex manifold of dimension $n \geq 2$ such that $$H^2(X) = H^{0,1}(X) = 0. $$ If any pluriharmonic function on $\Omega^c$ of class $C^\infty$ on $ \overline{\Omega^c}$ extends to a pluriharmonic function on $X$. Then for any $(1,1)$-differential form $f$ of class $C^\infty$ on $X$ with support in $\overline{\Omega}$ and $d$-closed, there exists a function $\check{v}$ of class $C^\infty$ on $X$ with support in $\overline{\Omega}$ such that $$\partial \bar{\partial}\check{v} = f . $$
\end{theo}
\begin{proof}
Let $f$ be a $(1,1)$-differential form of class $C^\infty$, $d$-closed on $X$ with support in $\overline{\Omega}$. Since $H^2(X) =0$, there exists a $1$-differential form $h$ of class $C^\infty$ and definite on $X$ such that $dh =f$.
We have $h= h_1+ h_2$ where $h_1$ and $h_2$ are respectively a $(0,1)$-differential form of class $C^\infty$ on $X$ and a $(1,0)$-differential form of class $C^\infty$ on $X$. We have $dh= dh_1 +dh_2=f$.
Since $d= \partial + \bar{\partial}$ and for reasons of bidegre we have $\partial h_2 =0$ and $\bar{\partial} h_1 =0$. Since $ H^{0,1}(X) = 0$, then there exist functions $g_1$ and $g_2$ of class $C^\infty$, definite on $X$ such that $h_1 = \bar{\partial}g_1$ and $h_2= \partial g_2$. We have
\begin{center}
$f = \partial h_1+ \bar{\partial}h_2= \partial \bar{\partial}g_1 +\bar{\partial}\partial g_2 = \partial \bar{\partial}(g_1 - g_2)$.
\end{center}
Let $v= g_1 - g_2$, $v$ is a function of class $C^\infty$, definite on $X$ such that $$\partial \bar{\partial} v = f.$$
Since $supp f \subset \overline{\Omega}$, then $X \setminus \overline{\Omega} \subset X \setminus suppf$. So $v$ is pluriharmonic on $\Omega^c$. By the extension property, $v$ extends into a pluriharmonic function $\tilde{v}$ on $X$. Let $$ \check{v} = v - \tilde{v}.$$
We have $\check{v}$ is a function of class $C^\infty$ with support in $\overline{\Omega}$ and $$\partial \bar{\partial}\check{v} = f .$$
\end{proof}
According to the Dolbeault and Poincare isomorphism, we have $$ H^{0,1}(X) = H^{0,1}_{L^p_{loc}}(X) = H^{0,1}_{C^k}(X) = H^{0,1}_{cur}(X) = 0$$ and 

$$H^{2}(X) = H^{2}_{L^p_{loc}}(X) = H^{2}_{C^k}(X) = H^{2}_{cur}(X) = 0.$$
We have the $L^p_{loc}$ version of Theorem \ref{b}
\begin{theo}
Let $X$ be a cohomologically $(n-1)$-complete complex manifold of dimension $n \geq 2$ such that $$H^2(X) = H^{0,1}(X) = 0. $$ If any pluriharmonic function on $\Omega^c$ which is the restriction on $\Omega^c$ of a function $L^p_{loc}(X)$ extends to a pluriharmonic function on $X$. Then for any $(1,1)$-differential form $f$ with coefficient in $L^p_{loc}(X)$, with support in $\overline{\Omega}$ and $d$-closed, there exists a function $\check{v} \in L^p_{loc}(X)$ with support in $\overline{\Omega}$ such that $$\partial \bar{\partial}\check{v} = f .$$
\end{theo}
We have the $C^k$ version with $k \geq 0$ of the Theorem \ref{b}
\begin{theo}
Let $X$ be a cohomologically $(n-1)$-complete  complex manifold of dimension $n \geq 2$ such that $$H^2(X) = H^{0,1}(X) = 0.$$ If any pluriharmonic function on $\Omega^c$ of class $C^k$ on $ \overline{\Omega^c}$ extends to a pluriharmonic function on $X$. Then for any $(1,1)$-differential form $f$ of class $C^k$ on $X$ with support in $\overline{\Omega}$ and $d$-closed, there exists a function $\check{v}$ of class $C^k$ on $X$ with support in $\overline{\Omega}$ such that $$\partial \bar{\partial}\check{v} = f . $$
\end{theo}
We have the distribution version of the Theorem \ref{b}
\begin{theo}
Let $X$ be a cohomologically $(n-1)$-complete  complex manifold of dimension $n \geq 2$ such that $$H^2(X) = H^{0,1}(X) = 0.$$ If any pluriharmonic function on $\Omega^c$ which is the restriction on $\Omega^c$ of a distribution on $X$ extends into a pluriharmonic function on $X$. Then for any $(1,1)$-current $f$ definite on $X$ with support in $\overline{\Omega}$ and $d$-closed, there exists a dustribution $\check{v}$ definite on $X$ with support in $\overline{\Omega}$ such that $$\partial \bar{\partial}\check{v} = f .$$
\end{theo}
\begin{coro}
Assume that $\Omega \subset X$ is relatively compact and $X$ a non-compact complex manifold such that $$ H^{2}_c(X) = H^{0,1}_{c}(X) =0. $$ If $\Omega^c$ is connected, then for any $(1,1)$-current $T$, $d$-closed with support in $\overline{\Omega}$, there exists a distribution $S$ with support in $\overline{\Omega}$ such that $$\partial \bar{\partial}S = T .$$
\end{coro}
\begin{proof}
Let $T$ be a $(1,1)$-current, $d$-closed with support in $\overline{\Omega}$. Since $ H^{2}_c(X) =0$, then there exists a $1$-current $u$ compactly supported in $X$ such that $du =T$. We have $u= u_1+ u_2$ where $u_1$ and $u_2$ are respectively a $(0,1)$-current with compact support in $X$ and a $(1,0)$-current with compact support in $X$. We have $du= du_1 +du_2=T$.
Since $d= \partial + \bar{\partial}$ and for bidegree reasons we have $\partial u_2 =0$ and $\bar{\partial} u_1 =0$. Since $ H^{0,1}_c(X) = 0$, then there exist distributions $g_1$ and $g_2$ with compact support in $X$ such that $u_1 = \bar{\partial}g_1$ and $u_2= \partial g_2$. We have
\begin{center}
$T = \partial u_1+ \bar{\partial}u_2= \partial \bar{\partial}g_1 +\bar{\partial}\partial g_2 = \partial \bar{\partial}(g_1 - g_2)$.
\end{center}
Let $S= g_1 - g_2$, $u$ is a distribution with compact support in $X$ such that $$\partial \bar{\partial} S = T.$$
The support of $S$ does not contain $\Omega^c$ otherwise $X = \overline{\Omega} \cup Supp S$ would be compact, which is absurd by hypothesis. So according to Proposition \ref{e}, $Supp S \subset \overline{\Omega}$.
\end{proof}

\end{document}